\documentclass{amsart}

\usepackage{amssymb}
\usepackage{graphicx}

\newtheorem{theorem}{Theorem}[section]
\newtheorem{lemma}[theorem]{Lemma}
\newtheorem{cor}[theorem]{Corollary}

\theoremstyle{definition}

\newtheorem{example}[theorem]{Example}

\theoremstyle{remark}
\newtheorem{remark}[theorem]{Remark}

\numberwithin{equation}{section}

\newcommand{\abs}[1]{\left\lvert#1\right\rvert}
\newcommand{\norm}[1]{\left\lVert#1\right\rVert}
\newcommand{\ds}{\displaystyle}

\newcommand{\CI}{\mathcal{I}}
\newcommand{\CJ}{\mathcal{J}}

\newcommand{\CU}{\mathcal{U}}

\begin{document}

\title[Distribution of Values of Short Hybrid Exponential Sums]
{The Distribution of Values of Short Hybrid Exponential Sums on Curves over Finite Fields}

%    author one information
\author{Kit-Ho Mak}
\address{Department of Mathematics \\
University of Illinois at Urbana-Champaign \\
273 Altgeld Hall, MC-382 \\
1409 W. Green Street \\
Urbana, Illinois 61801, USA}
\email{mak4@illinois.edu}

%    author two information
\author{Alexandru Zaharescu}
\address{Department of Mathematics \\
University of Illinois at Urbana-Champaign \\
273 Altgeld Hall, MC-382 \\
1409 W. Green Street \\
Urbana, Illinois 61801, USA}
\email{zaharesc@math.uiuc.edu}


\thanks{The second author is supported by NSF grant number DMS - 0901621.}

\begin{abstract}
Let $p$ be a prime number, $X$ be an absolutely irreducible affine plane curve over $\mathbb{F}_p$, and $g,f\in\mathbb{F}_p(x,y)$. We study the distribution of the values of the hybrid exponential sums
\begin{equation*}
S_n = \sum_{\substack{P_i\in X, n<x(P_i)\leq n+H \\ y(P_i)\in\mathcal{J}}}\chi(g(P_i))\psi(f(P_i))
\end{equation*}
on $n\in\mathcal{I}$ for some short interval $\mathcal{I}$. We show that under some natural conditions the limiting distribution of the projections of the sum $S_n$, $n\in\mathcal{I}$ on any straight line through the origin is Gaussian as $p$ tends to infinity.
\end{abstract}

\maketitle

\section{Introduction}

Many sequences that arise in number theory have Gaussian distribution. A well-known family of sequences with Gaussian distribution can be obtained by Erd\"{o}s-Kac type results \cite{ErKa40} (see also \cite{GrSo07} for a more complete and recent account). For example, the number of distinct prime factors of an integer $n$ \cite{ErKa40}, of $\phi(n)$ \cite{ErPo85}, of the sum $a+b$ when $a$ and $b$ are given in some dense set \cite{EMS87}, of the number of points on an elliptic curve \cite{Liu06}, of the characteristic polynomial of the Frobenius acting on Drinfeld modules \cite{KuLi09}, and of polynomials of several variables \cite{Xio09} are all with Gaussian distribution. Another example that falls into this type is the $2$-rank of the Selmer groups of certain $2$-isogenies of some families of elliptic curves \cite{XiZa08, XiZa09}. A well-known unpublished result of Selberg on the distribution of values of Riemann-Zeta function $\zeta(s)$ on the critical line offers another type of Gaussian distribution result. In this paper we will present another family of sequences arising naturally in number theory, which have a Gaussian distribution, but do not fall into the types mentioned above.

In \cite{DaEr52}, Davenport and Erd\"{o}s studied the distribution of quadratic residues and non-residues. As a result they proved that the limiting distribution of the values of the incomplete character sum
\begin{equation} \label{chsum}
S_n=\sum_{n<x\leq n+H}\chi(x),
\end{equation}
where $\chi$ is the quadratic character modulo a large prime $p$, is Gaussian after a suitable normalization. More precisely, they showed that the number $N_p(\lambda)$ of integers $n\in\{0,\ldots,p-1\}$ for which $S_n\leq \lambda H^{\frac{1}{2}}$ satisfies
\begin{equation*}
\lim_{p\rightarrow\infty}\frac{N_p(\lambda)}{p}=\frac{1}{\sqrt{2\pi}}\int_{-\infty}^{\lambda} e^{-\frac{t^2}{2}}\,dt
\end{equation*}
for any fixed $\lambda$, if $H$ satisfies the growth conditions
\[\ds H\rightarrow\infty, \frac{\log{H}}{\log{p}}\rightarrow 0 \]
as $p$ tends to infinity.

In \cite{CCZ03}, the result of Davenport and Erd\"{o}s is generalized to the case of an $n$-dimensional sum of quadratic characters of the form
\begin{equation*}
S_H(x_1,\ldots,x_n)=\sum_{x_1<z_1\leq x_1+H}\cdots\sum_{x_n<z_n\leq x_n+H}\chi(z_1+\ldots+z_n)\,.
\end{equation*}
 In this paper, we will generalize the result of Davenport and Erd\"{o}s in another direction by regarding the sum $S_n$ in \eqref{chsum} as a special example in a more general class of incomplete hybrid exponential sums over an absolutely irreducible affine plane algebraic curve $X$ over the finite field $\mathbb{F}_p$,
\begin{equation}\label{def:Sn}
S_n = \sum_{\substack{P_i\in X, n<x(P_i)\leq n+H \\ y(P_i)\in\CJ}}\chi(g(P_i))\psi(f(P_i)).
\end{equation}
Here $\chi$ is a multiplicative character of $\mathbb{F}_p$, $\psi$ is an additive character of $\mathbb{F}_p$, $\CJ$ an interval, and $g,f \in\mathbb{F}_p(x,y)$ are rational functions. The sum \eqref{chsum} considered in \cite{DaEr52} corresponds to the case when $\chi$ is the quadratic character, $\psi$ is the trivial character, $X$ the affine line defined by $y=0$, and $g(x,y)=x$. In this paper, we prove that the limiting distribution of the values of most of these incomplete hybrid exponential sums is also Gaussian.

\section{Statements of Main Results}

Let $p$ be a large prime, and $X$ be an absolutely irreducible affine plane curve over $\mathbb{F}_p$, given by the equation $P(x,y)=0$, with $\deg_y(P(x,y))\geq 1$, where $\deg_y$ denotes the degree in $y$. Let $\chi,\psi$ be a \textit{nontrivial} multiplicative character and a \textit{nontrivial} additive character modulo $p$ respectively, $f,g\in\mathbb{F}_p(x,y)$ be two rational functions.

Let $\CJ=[\alpha p,\beta p)$ be an interval, where $0\leq \alpha < \beta \leq 1$. For simplicity, we assume that no two points on $X$ with their $y$-coordinates in $\CJ$ have the same $x$-coordinates. If $r$ denotes the number of $\mathbb{F}_p$-points on $X$, we let $P_1,\dots,P_r$ be the points on $X$ with their $y$-coordinates in $\CJ$, ordered by their $x$-coordinates in ascending order. We also let $H$ be an integer such that $1\leq H \leq p$, and $\CI\subseteq[0,p-1]$ an interval. Since $\mathbb{F}_p$-points on an affine curve is uniformly distributed (see for example Meyerson \cite{Mey81}, Fujiwara \cite{Fuj85}, or the authors \cite{MaZa10}), we have the following estimation of $r$,
\begin{equation*}
r = (\beta-\alpha)p + O(\sqrt{p}\log^2{p}).
\end{equation*}
More generally, the number of points $N$ on $X$ inside the rectangle $(n,n+H]\times\CJ$ is given by
\begin{equation}\label{est:N}
N = (\beta-\alpha)H + O(\sqrt{p}\log^2{p}),
\end{equation}
where $\abs{\CI}$ denotes the number of integers in $\CI$.

We are interested in the distribution of the values of the hybrid exponential sums \eqref{def:Sn} for $n\in\CI$ as $p$ tends to infinity. It is understood that the poles of $f,g$ are excluded from the sum.

We will show that the projections of $S_n$ on any fixed straight line through the origin are Gaussian. More precisely, fix an angle $\theta\geq0$ and consider the line $L_{\theta}$ given by the equation $y=x\tan\theta$. Let $p$, $\chi$, $\psi$, $\CI,\CJ$, $f,g$ as above, we form the exponential sums $S_n$ as in \eqref{def:Sn} for $n\in\CI$, and study its projection $u_n$ on $L_{\theta}$, normalized by the asymptotic number of points we sum, namely $((\beta-\alpha)H)^{\frac{1}{2}}$ by \eqref{est:N}. i.e.
\begin{equation}\label{def:un}
u_n = \frac{S_ne^{-i\theta}+\overline{S_n}e^{i\theta}}{2((\beta-\alpha)H)^{\frac{1}{2}}},
\end{equation}
for $n\in\CI$, and consider the sequence $\{u_n:n\in\CI\}$ on $L_{\theta}$. We will show that as $H$ and $p$ tends to infinity, the limiting distribution of the $u_n$ is Gaussian. The idea is to calculate the moments
\begin{equation}\label{def:Mk}
M_k = M_k(p,\chi,\psi,f,g,H,\CI,\theta) = \sum_{n\in\CI}u_n^k
\end{equation}
for $k\in\mathbb{N}$. Our result is the following.

\begin{theorem}\label{mainthm}
Let $p$, $X$, $\chi$, $\psi$, $\CI$, $\CJ$, $H$ be as above. Let $g,f\in\mathbb{F}_p(x,y)$ be two rational functions. Define $d_g,d_f$ to be the degrees of the denominators of $g$ and $f$ respectively. Suppose $f$ is not of the form
\[h^p-h+\text{(linear terms)}+Q(x,y)P(x,y)^b\]
for any nonzero integer $b$, rational functions $h\in\overline{\mathbb{F}}_p(x,y), Q\in\mathbb{F}_p(x,y)$, with $Q$ relatively prime to $P$, and any constant $C\in\mathbb{F}_p$ (in this paper, all the ``linear terms'' have coefficients in $\mathbb{F}_p$). Let $f=\frac{f_1}{f_2}$, with $f_1,f_2\in\mathbb{F}_p[x,y]$, and $f_1,f_2$ have no common factors, we also assume that
\begin{enumerate}
\item if $f$ is a polynomial, then $\deg{f}<p$. Write $f(x,y)=r_1(x)+r_2(x,y)$, where $r_1$ consists of all terms which do not depend on $y$. We further assume that either
    \begin{enumerate}
    \item $r_2$ is not of the form
    \[\text{(linear terms)}+Q(x,y)P(x,y)^b\]
    for any nonzero integer $b$, rational function $Q$ relatively prime to $P$, or
    \item if $r_2$ is of the above form, then $\deg{r_1}\geq 3$.
    \end{enumerate}
\item if $\deg{f_2}\geq 1$ (so that $f$ is not a polynomial), then $f_2$ is not a constant multiple of the $p$-th power of any polynomial in $\overline{\mathbb{F}}_p[x,y]$.
\end{enumerate}

Let $k$ be a positive integer, $H$, $k$ be small compared to $p$ (say $H,k=O(\log{p})$). Then if $k$ is odd, we have
\begin{equation}\label{Mkodd}
M_k = O(H^{\frac{k}{2}}(d^{2k}+2d^{k}d_g^{k}+2d^{k}d_f^{k})\sqrt{p}\log^{2k}{p}),
\end{equation}
and if $k$ is even,
\begin{multline}\label{Mkeven}
M_k = \frac{1}{2^k}\binom{k}{k/2}(k/2)!\abs{\CI}(1+O(k^2/H)) \\
 +O(2^{\frac{k}{2}}H^{\frac{k}{2}}(d^{2k}+2d^{k}d_g^{k}+2d^{k}d_f^{k})\sqrt{p}\log^{2k}{p}).
\end{multline}
\end{theorem}

The main term in \eqref{Mkeven} is
\begin{equation*}
\frac{1}{2^k}\binom{k}{k/2}(k/2)!\abs{\CI}=2^{-\frac{k}{2}}\cdot 1 \cdot 3 \cdot \ldots \cdot (k-1)\cdot\abs{\CI}.
\end{equation*}
As in Davenport and Erd\"{o}s \cite{DaEr52}, we write
\begin{equation*}
\mu_k=\begin{cases}
1 \cdot 3 \cdot \ldots \cdot (k-1), & \text{if $k$ is even},\\
0, & \text{if $k$ is odd}.
\end{cases}
\end{equation*}
Then from Theorem \ref{mainthm}, the next corollary follows immediately.

\begin{cor}\label{maincor}
Let $p$, $X$, $\chi$, $\psi$, $\CI$, $\CJ$, $H$, $g$, $f$ be as above. Suppose that $f$ is not of the form
\[\text{(linear terms)}+Q(x,y)P(x,y)^b\]
for nonzero integer $b$, $Q\in\mathbb{F}_p(x,y)$ relatively prime to $P(x,y)$, and subject to the conditions
\begin{enumerate}
\item $f$ is not a polynomial, or
\item $f$ is a polynomial, write $f(x,y)=r_1(x)+r_2(x,y)$, where $r_1$ consists of all terms which do not depend on $y$. We assume that either
    \begin{enumerate}
    \item $r_2$ is not of the form
    \[\text{(linear terms)}+Q(x,y)P(x,y)^b\]
    for any nonzero integer $b$, rational function $Q$ relatively prime to $P$, or
    \item if $r_2$ is of that form, then $\deg{r_1}\geq 3$.
    \end{enumerate}
\end{enumerate}

Suppose in addition that $H$ is any function of $p$ that tends to infinity with $p$ subjected to the following conditions:
\begin{gather*}
1 \leq H \leq p, \\
\lim_{p\rightarrow\infty}\frac{\log{H}}{\log{p}} =\lim_{p\rightarrow\infty}\frac{\log{d}}{\log{p}} =\lim_{p\rightarrow\infty}\frac{\log{d_g}}{\log{p}} =\lim_{p\rightarrow\infty}\frac{\log{d_f}}{\log{p}}=0, \\
\liminf_{p\rightarrow\infty}\frac{\log{\abs{\CI}}}{\log{p}}>\frac{1}{2}.
\end{gather*}
Then we have
\begin{equation*}
\lim_{p\rightarrow\infty}\frac{2^{k/2}M_k}{\abs{\CI}}=\mu_k.
\end{equation*}
\end{cor}

From this asymptotic behaviour of the moments, one can deduce that the distribution of our sums $S_n$ tends to the Gaussian distribution on $L_{\theta}$ as $p$ tends to infinity. We will give the argument in Section \ref{pfcorGD}.

\begin{cor}\label{corGD}
Suppose that the hypotheses of Corollary \ref{maincor} are satisfied. Then for any $\lambda\geq 0$, the number $G_p(\lambda)$ of integers $n\in\CI$ with $u_n\leq\lambda$ satisfies
\begin{equation*}
\lim_{p\rightarrow\infty}\frac{G_p(\lambda)}{\abs{\CI}}=\frac{1}{\sqrt{\pi}}\int_{-\infty}^{\lambda}e^{-t^2}\,dt.
\end{equation*}
\end{cor}

Several remarks about the distribution of the sum $S_n$ are in order.

\begin{remark}
If $\CJ$ is not chosen so that we have a one-to-one correspondence of $x$ and $y$-coordinates on a curve, we may still have Gaussian distribution for the $S_n$. For example, if $X$ is a hyperelliptic curve, and we choose $\CJ$ to be the whole interval $[0,p)$, then generically one $x$-coordinate on the curve corresponds to two $y$-coordinates. From Corollary \ref{corGD}, we have Gaussian distribution for $\CJ_1=[0,p/2)$, and also for $\CJ_2=[-p/2,0)$. After combining the two of them we will have Gaussian distribution for the whole interval $\CJ=[-p/2,p/2)$.
\end{remark}

\begin{remark}
Corollary \ref{corGD} is in some sense best possible with respect to the range of $\abs{\CI}$, and $f$ has to be non-linear. This is illustrated in the following examples.
\end{remark}

\begin{example}
Let $X$ be the diagonal defined by $x=y$, $\chi$ the quadratic character and $\psi(x)=e_p(x)$, where $e_p(x)=e^{2\pi ix/p}$. Let $g(x,y)=x, f(x,y)=xy$. All the assumptions on $\chi, \psi, g, f$ are satisfied, so we can conclude from Corollary \ref{corGD} the Gaussian distribution of the hybrid exponential sum if $\abs{\CI}>p^{\frac{1}{2}}$. However, if we let $\CI=\{1,\ldots,N\}$, with $N\sim p^{\frac{1}{2}-\varepsilon}$ and $H<p^{\frac{1}{2}-\varepsilon}$, then if $p$ is large enough, $e_p(xy)\sim 1$. Since $\chi(x)$ is real, the sum $S_n$ will be close to a real number for any $n\in\CI$. Thus their projections to the imaginary axis will not have Gaussian distribution.
\end{example}

\begin{example}
On the other hand, if $X, \chi, \psi$ is as above, and let $g(x,y)=x$, and $f(x,y)=x+y$ is linear. Let $\CI=\{1,\ldots,N\}$ with $N, H = o(p)$ but $N>p^{\frac{1}{2}}$. Then again $e_p(x+y)\sim 1$ for large $p$, and by the same reason as in the above paragraph, the projections of $S_n$ to the imaginary axis will not have Gaussian distribution.
\end{example}

Note that our assumptions of $\chi$ and $\psi$ being non-trivial exclude us from considering sums like \eqref{chsum} appeared in \cite{DaEr52}. Our next goal is to extend our results to the cases when one of the characters $\chi$ or $\psi$ is trivial. For trivial $\chi$ we have the following.

\begin{theorem}\label{thmtrivialchi}
Theorem \ref{mainthm}, Corollary \ref{maincor} and Corollary \ref{corGD} remains true if $\chi$ is trivial but all other conditions are still assumed.
\end{theorem}

The case for trivial $\psi$ is more difficult, but we still obtain a Gaussian distribution if we impose the reasonable assumption that $g(x,y)$ is not a complete $a$-th power.

\begin{theorem}\label{thmtrivialpsi}
Let $p$, $X$, $\chi$, $\CI$, $\CJ$, $H$ be as in Theorem \ref{mainthm} and Corollary \ref{maincor}, and let $\psi$ be trivial. Let $a$ be the order of $\chi$. Assume $g(x,y)$ is not of the form
\[\ds h^a+Q(x,y)P(x,y)^b \]
for any nonzero integer $b$, $Q\in\mathbb{F}_p(x,y)$ relatively prime to $P$, $h\in\overline{\mathbb{F}}_p(x,y)$, and $\deg{(g)}$ is small compared to $p$. Let $k$, $H$ be positive integers which are small compared to $p$. Then we have the following.
\begin{enumerate}
\item If $a=2$, and $\theta=0$ (that is we only consider the distribution on the real line), we have
\begin{equation}\label{2Mkodd}
M_k = O(H^{\frac{k}{2}}(d^{2k}+2d^{k}d_g^{k}+2d^{k}d_f^{k})\sqrt{p}\log^{2k}{p})
\end{equation}
when $k$ is odd, and
\begin{multline}\label{2Mkeven}
M_k = \frac{k!}{2^{\frac{k}{2}}\left(\frac{k}{2}\right)!}\abs{\CI}(1+O(k^2/H)) \\
+O(2^{\frac{k}{2}}H^{k}(d^{2k}-2d^{k}d_{g}^{\frac{k}{2}}+2d^{k}d_{f}^{k})\sqrt{p}(2\log{p}+1)^{k})
\end{multline}
when $k$ is even.

\item If $a>2$ is even, we have
\begin{equation}\label{aMkodd}
M_k = O(H^{\frac{k}{2}}(d^{2k}+2d^{k}d_g^{k}+2d^{k}d_f^{k})\sqrt{p}\log^{2k}{p})
\end{equation}
when $k$ is odd, and
\begin{multline}\label{aMkeven}
M_k = \frac{1}{2^k}\frac{k!}{(k/2)!}\abs{\CI} (1+O(k^{\frac{a}{2}+2}/H)) \\
 +O(H^{\frac{3k}{2}}(d^{4k}-2d^{2k}d_{g}^{k}+2d^{2k}d_{f}^{2k})\sqrt{p}\log^{2k}{p})
\end{multline}
when $k$ is even.

\item If $a$ is odd (necessarily $a>1$ since $\chi$ is nontrivial), we have
\begin{multline}\label{oMkodd}
M_k = \frac{1}{2^k((\beta-\alpha)H)^{\frac{a}{2}-1}}\frac{k!}{\left(\frac{k}{2}+\frac{a}{2}\right)!}(2\cos{a\theta})\abs{\CI}(1+O(k^{a+2}/H)) \\
 +O((H^{\frac{3k}{2}}(d^{4k}-2d^{2k}d_{g}^{k}+2d^{2k}d_{f}^{2k})\sqrt{p}\log^{2k}{p}).
\end{multline}
when $k$ is odd, and
\begin{multline}\label{oMkeven}
M_k = \frac{1}{2^k}\frac{k!}{(k/2)!}\abs{\CI} (1+O(k^{\frac{a}{2}+2}/H)) \\
 +O(H^{\frac{3k}{2}}(d^{4k}-2d^{2k}d_{g}^{k}+2d^{2k}d_{f}^{2k})\sqrt{p}\log^{2k}{p})
\end{multline}
when $k$ is even.
\end{enumerate}
\end{theorem}

The analogue to Corollary \ref{maincor} for trivial $\psi$ is the following.
\begin{cor}\label{maincortrivialpsi}
Let
\begin{equation*}
\mu_k=\begin{cases}
1 \cdot 3 \cdot \ldots \cdot (k-1), & \text{if $k$ is even},\\
0, & \text{if $k$ is odd}.
\end{cases}
\end{equation*}
If $\psi$ is trivial, and keeping the other assumptions in Theorem \ref{thmtrivialpsi} and Corollary \ref{maincor}, then if we take the limit as $p$ tends to infinity with $\chi$ being a series of quadratic characters modulo $p$, and we only consider the moments on the real line, then
\begin{equation*}
\lim_{p\rightarrow\infty}\frac{M_k}{\abs{\CI}}=\mu_k.
\end{equation*}

If on the other hand, we take the limit with $\chi$ being restricted to characters of order $a>2$, then the same conclusion as in Corollary \ref{maincor} holds. That is,
\begin{equation*}
\lim_{p\rightarrow\infty}\frac{2^{k/2}M_k}{\abs{\CI}}=\mu_k.
\end{equation*}
\end{cor}

Finally, we get the analogue to Corollary \ref{corGD}, which shows that we still have Gaussian distribution when $\psi$ is trivial.
\begin{cor}
Assumptions as in Corollary \ref{maincortrivialpsi}. For any $\lambda\geq 0$, let $G_p(\lambda)$ be the number of integers $n\in\CI$ with $u_n\leq\lambda$. If $p$ tends to infinity, with $\chi$ a quadratic character modulo $p$ and we only consider the distribution on the real line, then
\begin{equation*}
\lim_{p\rightarrow\infty}\frac{G_p(\lambda)}{\abs{\CI}}=\frac{1}{\sqrt{2\pi}}\int_{-\infty}^{\lambda}e^{-\frac{t^2}{2}}\,dt.
\end{equation*}
On the other hand, if we restrict the $\chi$ to be characters of order $a>2$, then the same conclusion as Corollary \ref{corGD} holds. That is,
\begin{equation*}
\lim_{p\rightarrow\infty}\frac{G_p(\lambda)}{\abs{\CI}}=\frac{1}{\sqrt{\pi}}\int_{-\infty}^{\lambda}e^{-t^2}\,dt.
\end{equation*}

Therefore, we have Gaussian distributions in all the above cases, but when we take the limit through a series of quadratic character modulo $p$, we get a Gaussian distribution with different parameters compared to all other cases.
\end{cor}

\begin{remark}\label{rmkreal}
If the order of $\chi$ is $a=2$, then we only have Gaussian distribution on the real line, but not when $S_n$ is projected to other lines. The reason is simple: since $\chi$ is quadratic, our $S_n$ is real for any $n$ in this case, and we certainly do not have Gaussian distribution if we project $S_n$ to the imaginary axis.
\end{remark}

\begin{remark}
Although we get the same results for odd and even $a$, we note that when $a$ is odd, our estimation shows that we just barely obtain the Gaussian distribution. In fact, the main term \eqref{oMkodd} for the case $a$ and $k$ both odd is of order $\abs{\CI}/H$, which just barely tends to zero after dividing by $\abs{\CI}$, thanks to the assumption that $H$ tends to infinity with $p$.
\end{remark}

\begin{remark}
We can get back the result from Davenport-Erd\"{o}s \cite{DaEr52} if we take $X$ to be the straight line $y=0$, $\chi$ being the quadratic character modulo $p$, $\psi$ trivial, and $g(x)=x$. Note that this is exactly the case when we get different parameters for the Gaussian distribution.
\end{remark}

\section{Some preliminaries}

To prove Theorem \ref{mainthm}, the first thing we need is an estimation for the incomplete hybrid exponential sums over an affine space curve $Y\subseteq\mathbb{A}^m$, which need not be irreducible nor reduced. The sum is defined as follows.
\begin{equation*}
S_{\CJ_1,\ldots,\CJ_m} = \sum_{\textbf{x}\in Y\cap(\CJ_1\times\ldots\times\CJ_m)}\chi(\tilde{g}(\textbf{x}))\psi(\tilde{f}(\textbf{x})),
\end{equation*}
where $\textbf{x}=(x_1,\ldots,x_m)$ and $\CJ_i\subseteq[0,p-1]$ are intervals, $\tilde{g}, \tilde{f}\in\mathbb{F}_p(x_1,\ldots,x_m)$ are rational functions, and $\psi$ is a nontrivial character.

\begin{lemma}\label{lem:perel}
Let $p$ be a large prime, $D$ be the degree of $Y$, $d_{\tilde{g}},d_{\tilde{f}}$ the degrees of the denominators of $\tilde{g}, \tilde{f}$ respectively. Let $a$ be the order of $\chi$. Unless there are rational functions $\tilde{g_1},\tilde{f_1}\in\overline{\mathbb{F}}_p(x_1,\ldots,x_m)$ such that $\tilde{g}-\tilde{g_1}^a$ vanishes identically and $\tilde{f}-\tilde{f_1}^p+\tilde{f_1}$ is linear on some irreducible component of $Y$ simultaneously, we have
\begin{equation*}
\abs{S_{\CJ_1,\ldots,\CJ_m}}\leq ((D^2-3D+2Dd_{\tilde{g}}+2Dd_{\tilde{f}})\sqrt{p}+D^2+O(D))(2\log{p}+1)^m.
\end{equation*}
\end{lemma}
\begin{proof}
The work of Perel'muter \cite{Per69} deals with the case when $S_{\CJ_1,\ldots,\CJ_m}$ is complete, i.e. if all $\CJ_i=[0,p-1]$. He showed that unless $\tilde{g}-\tilde{g_1}^a$ and $\tilde{f}-f_1^p+f_1$ vanishes identically on some irreducible component of $Y$ simultaneously, the complete sum satisfies
\begin{equation}\label{comsum}
\abs{S_{[0,p-1]^m}} \leq (D^2-3D+2Dd_{\tilde{g}}+2Dd_{\tilde{f}})\sqrt{p}+D^2+O(D).
\end{equation}
His work uses the idea of Bombieri-Weil type estimate of an exponential sum along an algebraic curve \cite{Bom66, Wei48}. Note that compared to \cite{Per69} we have an extra $O(D)$ term because we are considering an affine curve, thereby missing at most $O(D)$ terms in the sum, each of those having absolute value at most $1$.

We then express our incomplete sum $S_{\CJ_1,\ldots,\CJ_m}$ in terms of complete sums of the same type. Recall the orthogonal relation
\begin{equation}\label{orthrel}
\frac{1}{p}\sum_{t\text{~mod~$p$}}\psi(ty) = \begin{cases}
1, & \text{if~} y=0, \\
0, & \text{otherwise,}
\end{cases}
\end{equation}
we see that
\begin{align}
S_{\CJ_1,\ldots,\CJ_m} &= \sum_{\textbf{x}\in Y}\chi(\tilde{g}(\textbf{x}))\psi(\tilde{f}(\textbf{x}))\prod_{i=1}^m\left(\sum_{m_i\in\CJ_i}\frac{1}{p}\sum_{t_i\text{~mod~$p$}}\psi(t_i(m_i-x_i))\right) \label{eqnSJ}\\
&= \frac{1}{p^m}\prod_{i=1}^m\sum_{t_i\text{~mod~$p$}}\left(\sum_{m_i\in\CJ_i}\psi(t_im_i)\right) \nonumber\\
&\qquad\times\sum_{\textbf{x}\in Y}\chi(\tilde{g}(\textbf{x}))\psi(\tilde{f}(\textbf{x})-t_1x_1-\ldots-t_mx_m). \nonumber
\end{align}

From the assumption in our lemma, we see that the inner sum satisfies the assumption in \cite{Per69}, and so can be estimated by \eqref{comsum}. To estimate the outer sum, first we need the estimation
\begin{equation}\label{est1}
\abs{\sum_{t\text{~mod~$p$}}\left(\sum_{m\in\CJ}\psi(tm)\right)} \leq 2p\log{p}+\abs{\CJ}.
\end{equation}
To see this, note that any nontrivial additive character $\psi$ modulo $p$ is of the form $\psi(x)=e_p(kx)=e^{2\pi i kx/p}$ ($x\in\mathbb{F}_p$) for some $k$ with $(k,p)=1$. Let $\CJ\cap\mathbb{Z} = \{l,l+1,\dots,l+h-1\}$, where $h=\abs{\CJ}$, then
\begin{equation*}
\sum_{m\in\CJ}\psi(tm) = \sum_{m\in\CJ}e_p(ktm) =
\begin{cases}
h & \text{if}~ t=0, \\
\left(e^{\frac{-2\pi itkl}{p}}\right)\frac{1-e^{-2\pi itkh/p}}{1-e^{-2\pi itk/p}} & \text{if}~ t\neq 0.
\end{cases}
\end{equation*}
Hence if $t\neq 0$,
\begin{equation*}
\abs{\sum_{m\in\CJ}\psi(tm)} \leq \frac{2}{\abs{1-e^{-2\pi itk/p}}}.
\end{equation*}
If $\norm{\cdot}$ denotes the distance to the nearest integer, then $\abs{1-e^{-2\pi itk/p}} = 2\abs{\sin{(\pi tk/p)}} \geq \norm{\frac{kt}{p}}$ for $p$ large enough. We obtain the estimate
\begin{equation*}
\abs{\sum_{m\in\CJ}\psi(tm)} \leq 2\left(\norm{\frac{kt}{p}}\right)^{-1}.
\end{equation*}
We then sum the above over all $t$ modulo $p$. We choose the set of representatives with $0\leq\abs{t}\leq (p-1)/2$, noting that for $t\neq 0$, $(k,p)=1$, $\norm{\frac{kt}{p}}$ is a reordering of $\norm{\frac{t}{p}}$, but in our set of representatives, $\norm{\frac{t}{p}}=\frac{\abs{t}}{p}$. Now \eqref{est1} follows from the elementary inequality
\begin{equation*}
1+\frac{1}{2}+\dots+\frac{1}{\frac{p-1}{2}} \leq \log{p}.
\end{equation*}
Finally, putting \eqref{comsum} and \eqref{est1} into \eqref{eqnSJ}, we get
\begin{align*}
\abs{S_{\CJ_1,\ldots,\CJ_m}} &\leq \frac{1}{p^m}\prod_{i=1}^m (2p\log{p}+\abs{\CJ_i})((D^2-3D+2Dd_{\tilde{g}}+2Dd_{\tilde{f}})\sqrt{p}+D^2+O(D)) \\
&\leq ((D^2-3D+2Dd_{\tilde{g}}+2Dd_{\tilde{f}})\sqrt{p}+D^2+O(D))(2\log{p}+1)^m.
\end{align*}
\end{proof}

\begin{remark}\label{rmkperel}
If we assume that $\tilde{g}-\tilde{g_1}^a$ is not identically zero on $X$, then the above lemma still hold even when $\psi$ is the trivial character. In fact, the same proof hold by using any arbitrarily chosen nontrivial $\psi$ for \eqref{orthrel}. This remark will be useful when we prove Theorem \ref{thmtrivialpsi} for sums with trivial $\psi$.
\end{remark}

\begin{remark}
If $\tilde{g}-\tilde{g_1}^a$ vanishes identically and $\tilde{f}-\tilde{f_1}^p+\tilde{f_1}$ is linear on some irreducible component of $Y$ simultaneously, the resulting hybrid sum may be large in some interval $\CJ_i$.

For example, let $Y$ be the elliptic curve defined by the equation $y^2-x^3-ax-b=0$, $\CJ=[0,p/2)$ and $\chi$ the quadratic character of $\mathbb{F}_p$. Suppose now $g(x,y)=x^2$ and $f(x,y)=x^p-x$, so that $\chi(g(x,y))=1$ and $\psi(f(x,y))=1$ for any $\mathbb{F}_p$-point $(x,y)$. Then each term in the hybrid sum is $1$, and hence if $\abs{\CI}>p^{\frac{1}{2}}$, we will have
\begin{equation*}
S_{\CI,\CJ} = \frac{1}{2}\abs{\CI} + O(\sqrt{p}),
\end{equation*}
which is much bigger than the bound suggested in Lemma \ref{lem:perel} when $p$ is large.
\end{remark}

The following strange looking lemma prove that certain rational functions are not of the form disallowed by Lemma \ref{lem:perel}. This will be of vital importance for our later calculations.

\begin{lemma}\label{lem:f}
Let $p$ be a large prime, $f\in\mathbb{F}_p(x,y)$ be a rational function in two variables, $f=f_1/f_2$, $f_1,f_2\in\mathbb{F}_p[x,y]$, $f_1,f_2$ has no common factors. Suppose that $f\neq h^p-h+\text{(linear terms)}$ for any rational function $h\in\overline{\mathbb{F}}_p(x,y)$, and subject to the following conditions:
\begin{enumerate}
\item If $f$ is a polynomial, then $\deg{f}<p$. Write $f(x,y)=r_1(x)+r_2(x,y)$, where $r_1$ consists of all terms which do not depend on $y$. We further assume that either $r_2$ is not linear, or if $r_2$ is linear, then $\deg{r_1}\geq 3$.
\item If $\deg{f_2}\geq 1$ (so that $f$ is not a polynomial), then $f_2$ is not a constant multiple of the $p$-th power of any polynomial in $\mathbb{F}_p[x,y]$.
\end{enumerate}

Let $H$, $j_1,j_2$ be positive integers so that both $H$ and $j_1+j_2$ are small compared to $p$. Let $1\leq h_1,\ldots, h_{j_1+j_2}\leq H$ be integers, which may or may not be distinct. Let $y_1,\ldots,y_{j_1+j_2}$ be indeterminates, which again may or may not be the same. Suppose that $y_i,y_j$ stand for the same indeterminate if and only if $h_i=h_j$. Define
\[F(x,y_1,\ldots,y_{j_1+j_2})=\sum_{j=1}^{j_1}f(x+h_j,y_j)-\sum_{j=j_1+1}^{j_1+j_2}f(x+h_j,y_j).\]
Then if $F=\tilde{h}^p-\tilde{h}+\text{(linear terms)}$ for some rational function $\tilde{h}\in\overline{\mathbb{F}}_p(x,y_1,\ldots,y_{j_1+j_2})$, we have $j_1=j_2$ and $F(x,y_1,\ldots,y_{j_1+j_2})$ is the zero polynomial.
\end{lemma}
\begin{proof}
First we collect the terms in $F$ that coincide (i.e. with equal $h_j$'s) and reordering the $y_j$'s if necessary, we get
\begin{equation}\label{eqnF}
F=m_1f(x+u_1,y_1)+\ldots+m_rf(x+u_r,y_r),
\end{equation}
where $m_1,\ldots,m_r\in\mathbb{Z}$, $u_1,\ldots,u_r\in\mathbb{F}_p$ are distinct, $1\leq u_i \leq H$, and $y_1,\ldots,y_r$ are distinct indeterminate. It suffices to show that $m_1,\ldots,m_r$ are all zero.

Suppose not all the $m_j$'s are zero, then by removing the $m_j$'s that are zero, we may assume that $m_j\neq 0$ for any $j$ in \eqref{eqnF}. By assumption, $F=\tilde{h}^p-\tilde{h}+\text{(linear terms)}$ for some rational function $\tilde{h}$.

First, if $f$ is a polynomial, then $F$ and hence $h$ are polynomials. From \eqref{eqnF} and the assumption that $\deg{f}<p$, we see that $\deg{F}<p$. However, if $\tilde{h}$ is non-constant, then $F=\tilde{h}^p-\tilde{h}+\text{(linear terms)}$ has degree greater than or equal to $p$, which is impossible. So $\tilde{h}$ is a constant, and so $F-\text{(linear terms)}=\tilde{h}^p-\tilde{h}\in\mathbb{F}_p$. This implies $F$ is linear. We claim that this is also impossible unless $F$ is zero.

To prove the claim, we let $f(x,y)=r_1(x)+r_2(x,y)$, where $r_1$ consists of all the terms that do not depend on $y$. From \eqref{eqnF}, we see that
\[F=(m_1r_1(x+u_1)+\ldots+m_rr_1(x+u_r)) + (m_1r_2(x,y_1)+\ldots+m_rr_2(x+u_r,y_r)) \]
is linear. This clearly implies that
\begin{align*}
R_1(x) &= m_1r_1(x+u_1)+\ldots+m_rr_1(x+u_r), \\
R_2(x,y_1,\ldots,y_r) &= m_1r_2(x,y_1)+\ldots+m_rr_2(x+u_r,y_r)
\end{align*}
are both linear. From the expression for $R_2$ it is immediate that $r_2$ is linear, and the conditions that $H, j_1+j_2$ is small compared to $p$ and $\deg{f}<p$ ensure that $\deg{r_1}\leq 2$ (or otherwise the coefficient of $x^{\deg{r_1}-1}$ in $R_1(x)$ does not vanish and so it cannot be linear). This contradicts to our assumption imposed on $r_1$ and $r_2$.

On the other hand, if $\deg{f_2}\geq 1$, then let $F=F_1/F_2, \tilde{h}=h_1/h_2$ be in lowest form (the numerator has no common factors with the denominator). It is easy to see that $\deg{h_2}\geq 1$. By clearing the denominator in \eqref{eqnF} and compare with $F=h^p-h+\text{(linear terms)}$, we get
\begin{equation*}
\frac{F_1}{f_2(x+u_1,y_1)\ldots f_2(x+u_r,y_r)} = \frac{h_1^p-h_1h_2^{p-1}-h_2^p\text{(linear terms)}}{h_2^p}.
\end{equation*}
Both sides are clearly in its lowest form. Hence $f_2(x+u_1,y_1)\ldots f_2(x+u_r,y_r)=h_2^p$, which implies each of the $f_2(x+u_j,y_j)$ is a constant multiple of a complete $p$-th power (here the fact that $j_1+j_2$ is small compared to $p$ is critical, so that the factors of $f_2$ that involves only $x$ cannot stack together and become a $p$-th power if they are not originally a $p$-th power). This is a contradiction to our assumption when $\deg{f_2}\geq 1$.
\end{proof}

\section{Computation of the moments $M_k$}

Recall that $S_n$ is defined by \eqref{def:Sn}, $u_n$ by \eqref{def:un} and the moments $M_k$ by \eqref{def:Mk}. Our calculation of $M_k$ starts with
\begin{align}
M_k &= \frac{1}{2^k((\beta-\alpha)H)^{\frac{k}{2}}}\sum_{n\in\CI}\sum^k_{j=0}\binom{k}{j}e^{-ji\theta+(k-j)i\theta}S_n^j\overline{S_n}^{k-j} \label{eqnMk}\\
&= \frac{1}{2^k((\beta-\alpha)H)^{\frac{k}{2}}}\sum^k_{j=0}\binom{k}{j}e^{(k-2j)i\theta}S(j,k-j), \nonumber
\end{align}
where
\begin{equation}\label{eqnSJ12}
S(j_1,j_2) = \sum_{n\in\CI}S_n^{j_1}\overline{S_n}^{j_2}.
\end{equation}
The diagonal sum $S(j,j)$ behave differently from the non-diagonal ones, and we treat them separately.

\subsection{The sum $S(j,j)$}

For $j\geq 0$ we have
\begin{equation*}
S(j,j)=\sum_{n\in\CI}\abs{S_n}^{2j},
\end{equation*}
and clearly $S(0,0)=\abs{\CI}+O(1)$. An estimate for $S(j,j)$ when $j$ is positive is given by the following lemma.

\begin{lemma}\label{lem:SJJ}
Let $p$ be a large prime, and $X$ be an irreducible affine plane curve of degree $d>1$ over $\mathbb{F}_p$ defined by the equation $P(x,y)=0$, $\chi,\psi$ be a nontrivial multiplicative character and a nontrivial additive character modulo $p$ respectively, $f,g\in\mathbb{F}_p(x,y)$ be two rational functions. Define $d_g,d_f$ be the degree of the denominator of $g$ and $f$ respectively. Suppose $f$ satisfies the same conditions as in Theorem \ref{mainthm}.

Let $\CI\subseteq[0,p-1]$ an interval and $\CJ=[\alpha p,\beta p)$ an interval, where $0\leq \alpha < \beta \leq 1$, such that no two points on $X$ with their $y$-coordinates in $\CJ$ have the same $x$-coordinates. Let $H$, $j$ be small compared to $p$, then we have
\begin{equation*}
S(j,j) = j!H^j\abs{\CI}(\beta-\alpha)^{2j}(1+O(j^2/H))+O(H^{2j}(d^{4j}-2d^{2j}d_{g}^j+2d^{2j}d_{f}^{2j})\sqrt{p}\log^{2j}{p}).
\end{equation*}
\end{lemma}
\begin{proof}
We have
\begin{align}\label{calSn}
\abs{S_n}^{2j} &= \sum_{\substack{n<x(P_{i_1})\leq n+H \\ y(P_{i_1})\in\CJ}}\cdots\sum_{\substack{n<x(P_{i_{2j}})\leq n+H \\ y(P_{i_{2j}})\in\CJ}}\prod_{l=1}^j\chi(g(P_{i_l}))\psi(f(P_{i_l})) \\
&\qquad\prod_{l=j+1}^{2j}\bar{\chi}(g(P_{i_l}))\bar{\psi}(f(P_{i_l})) \nonumber \\
&= \sum_{\substack{n<x(P_{i_1})\leq n+H \\ y(P_{i_1})\in\CJ}}\cdots\sum_{\substack{n<x(P_{i_{2j}})\leq n+H \\ y(P_{i_{2j}})\in\CJ}}\chi\left(\frac{g(P_{i_1})\ldots g(P_{i_j})}{g(P_{i_{j+1}})\ldots g(P_{i_{2j}})}\right) \nonumber \\
&\qquad \times\psi\left(\sum_{l=1}^j f(P_{i_l})-\sum_{l=j+1}^{2j} f(P_{i_l})\right). \nonumber
\end{align}
The main difficulty here is that the contents inside the two characters are not rational functions, and so Lemma \ref{lem:perel} is not directly applicable. We proceed by transforming the sum into a hybrid sum on another curve, so that we can apply Lemma \ref{lem:perel}.

If $X$ be an absolutely irreducible affine plane curve defined by $P(x,y)=0$, and $\CU=\{u_1,\ldots,u_m\}$ be a subset of $\{1,\ldots,p\}$. Similar to \cite{MaZa10}, to each pair $(X,\CU)$, we define the $x$-shifted curve of $X$ by $\CU$, $X_{\CU}$, to be the curve defined by the family of equations
\begin{align*}
P(x+u_1,y_1) &= 0 \\
P(x+u_2,y_2) &= 0 \\
\vdots & \\
P(x+u_m,y_m) &= 0
\end{align*}
in $\mathbb{A}_p^{m+1}$, the affine $(m+1)$-space over $\mathbb{F}_p$. It is easy to see that $X_{\CU}$ is indeed a curve. (Note that the definition here is a little bit different from that of \cite{MaZa10}.) From the definition of $C_{\CU}$ it is immediate that a point $(x,y_1,\ldots,y_m)$ of $X_{\CU}$ correspond to an $m$-tuple $(Q_1,\ldots,Q_m)$ of distinct points in $X$ with $x(Q_i)=x+u_i$.

Now fix a $(2j)$-tuple $\textbf{h}=(h_1,\ldots,h_{2j})$, with $1\leq h_i\leq H$, and set $\CU_{\textbf{h}}=\{u_1,\ldots,u_m\}$ be the set of all $h_i$ without multiplicity. By our assumption on $\CJ$, $P_{i_{l_1}}=P_{i_{l_2}}$ if and only if $h_{l_1}=h_{l_2}$. Thus we can view the $(2j)$-tuple $(P_{i_1},\ldots,P_{i_{2j}})$ appeared in the above sum \eqref{calSn} as a point on $X_{\CU}$ with $x$-coordinates equal to $n$, and this correspondence is one-to-one. So using \eqref{calSn}, and change the order of summation, we get
\begin{align*}
S(j,j) &= \sum_{h_1=1}^H\cdots\sum_{h_{2j}=1}^H\sum_{\substack{x\in\CI, y_1,\ldots,y_{2j}\in\CJ \\ (x,y_1,\ldots,y_{2j})\in X_{\CU_{\textbf{h}}}}}\chi\left(\frac{g(x+h_1,y_1)\ldots g(x+h_j,y_j)}{g(x+h_{j+1},y_{j+1})\ldots g(x+h_{2j},y_{2j})}\right) \\
&\qquad \times\psi\left(\sum_{l=1}^j f(x+h_l,y_l)-\sum_{l=j+1}^{2j} f(x+h_l,y_l)\right),
\end{align*}
where $y_i$ and $y_j$ stand for the same indeterminate if and only if $h_i=h_j$. Since
\begin{equation*}
\tilde{g}_{\textbf{h}}(x,y_1,\ldots,y_{2j}) = \frac{g(x+h_1,y_1)\ldots g(x+h_j,y_j)}{g(x+h_{j+1},y_{j+1})\ldots g(x+h_{2j},y_{2j})}
\end{equation*}
and
\begin{equation*}
\tilde{f}_{\textbf{h}}(x,y_1,\ldots,y_{2j}) = \sum_{l=1}^j f(x+h_l,y_l)-\sum_{l=j+1}^{2j} f(x+h_l,y_l)
\end{equation*}
are rational functions, we can now apply Lemma \ref{lem:perel} whenever the assumptions in that lemma are satisfied. We first calculate
\begin{gather*}
D=\deg{X_{\CU_{\textbf{h}}}}\leq d^{2j}, \\
\deg{\text{(denominator of $\tilde{g}_{\textbf{h}}$)}} \leq d_g^j, \\
\deg{\text{(denominator of $\tilde{f}_{\textbf{h}}$)}} \leq d_f^{2j}.
\end{gather*}

To estimate the number of $(2j)$-tuples $\textbf{h}=(h_1,\ldots,h_{2j}$ that does not satisfy the assumption of Lemma \ref{lem:perel} is to estimate the number of such tuples with $\tilde{g}_{\textbf{h}}$ being an $($ord $\chi)$-th power and $\tilde{f}_{\textbf{h}}$ is of the form $h^p-h+\text{(linear terms)}$. From Lemma \ref{lem:f}, we must have $\tilde{f}_{\textbf{h}}=0$, and so $(h_1,\ldots,h_j)$ and $(h_{j+1},\ldots,h_{2j})$ only differs by a permutation. Since there are $j(j-1)/2$ possible pairs from a $j$-tuple, there are a total $O(j^2H^{j-1})$ $j$-tuples that have at least two equal components, and for $(h_1,\ldots,h_j)$ with distinct components, there are exactly $j!$ possible permutations. Thus, the number of terms that we cannot use Lemma \ref{lem:perel} to estimate is $j!H^j(1+O(j^2/H))$. By the fact that $\mathbb{F}_p$-points are uniformly distributed on an affine curve (see Corollary 2.7 in \cite{MaZa10}), each of the terms with distinct components contribute
\[\abs{\CI}(\beta-\alpha)^{j}+O(2^{j}d^{2j}\sqrt{p}\log^{j}{p}) \]
to the sum (except when we hit a pole of $g(x,y)$, and their contribution can be absorbed in the error term above), and is less for the terms with at least two components equal. Hence, the sum of all terms that we cannot apply Lemma \ref{lem:perel} is
\[j!H^j\abs{\CI}(\beta-\alpha)^{j}(1+O(j^2/H))+O(2^{j}d^{2j}\sqrt{p}\log^{j}{p}).\]
For the other terms, we use Lemma \ref{lem:perel}, and the contribution of these terms to the sum is
\[O(H^{2j}(d^{4j}-2d^{2j}d_{g}^j+2d^{2j}d_{f}^{2j})\sqrt{p}(2\log{p}+1)^{2j}).\]
Combining the above two estimations, we finally get
\begin{equation*}
S(j,j) = j!H^j\abs{\CI}(\beta-\alpha)^{j}(1+O(j^2/H))+O(2^jH^{2j}(d^{4j}-2d^{2j}d_{g}^j+2d^{2j}d_{f}^{2j})\sqrt{p}\log^{2j}{p}).
\end{equation*}
This finishes the proof of our lemma.
\end{proof}

\subsection{The sum $S(j_1,j_2)$ for $j_1\neq j_2$}

As in the previous subsection, fix a $(j_1+j_2)$-tuple $\textbf{h}=(h_1,\ldots,h_{j_1+j_2})$, with $1\leq h_i\leq H$, and set $\CU_{\textbf{h}}=\{u_1,\ldots,u_m\}$ be the set of all $h_i$ without multiplicity. We have
\begin{multline}\label{calSJ12}
S(j_1,j_2) = \sum_{h_1=1}^H\cdots\sum_{h_{j_1+j_2}=1}^H\sum_{\substack{x\in\CI, y_1,\ldots,y_{j_1+j_2}\in\CJ \\ (x,y_1,\ldots,y_{j_1+j_2})\in X_{\CU_{\textbf{h}}}}}\chi(\tilde{g}(x,y_1,\ldots,y_{j_1+j_2})) \\
\times\psi(\tilde{f}(x,y_1,\ldots,y_{j_1+j_2})),
\end{multline}
where
\begin{equation*}
\tilde{g}_{\textbf{h}}(x,y_1,\ldots,y_{j_1+j_2}) = \frac{g(x+h_1,y_1)\ldots g(x+h_{j_1},y_{j_1})}{g(x+h_{j_1+1},y_{j_1+1})\ldots g(x+h_{j_1+j_2},y_{j_1+j_2})}
\end{equation*}
and
\begin{equation*}
\tilde{f}_{\textbf{h}}(x,y_1,\ldots,y_{j_1+j_2}) = \sum_{l=1}^{j_1} f(x+h_l,y_l)-\sum_{l=j_1+1}^{j_1+j_2} f(x+h_l,y_l).
\end{equation*}
Here again $y_i$ and $y_j$ stand for the same indeterminate if and only if $h_i=h_j$. We also have
\begin{gather*}
D=\deg{X_{\CU_{\textbf{h}}}}\leq d^{j_1+j_2}, \\
\deg{\text{(denominator of $\tilde{g}_{\textbf{h}}$)}} \leq d_g^{j_2}, \\
\deg{\text{(denominator of $\tilde{f}_{\textbf{h}}$)}} \leq d_f^{j_1+j_2}.
\end{gather*}

Unlike the case for $S(j,j)$, here by Lemma \ref{lem:f} we see that every term in our sum satisfy the assumption in Lemma \ref{lem:perel}. Therefore, we easily get the following lemma.

\begin{lemma}\label{lem:SJ12}
Assumptions as in Lemma \ref{lem:SJJ}. We have
\begin{equation*}
S(j_1,j_2) = O(H^{j_1+j_2}(d^{2j_1+2j_2}+2d^{j_1+j_2}d_g^{j_2}+2d^{j_1+j_2}d_f^{j_1+j_2})\sqrt{p}\log^{2j_1+2j_2}{p}).
\end{equation*}
\end{lemma}

\section{The proof of Theorem \ref{mainthm}}

Now we have all the ingredients we need to calculate the moments $M_k$. First suppose $k$ is an odd positive integer. Then $j\neq k-j$ for any integer $j$, so we can bound $M_k$ by using Lemma \ref{lem:SJ12} in the formula \eqref{eqnMk}. We get
\begin{align*}
M_k &= O\left(\frac{1}{2^k((\beta-\alpha)H)^{\frac{k}{2}}}\sum^k_{j=0}\binom{k}{j}H^{k}(d^{2k}+2d^{k}d_g^{k-j}+2d^{k}d_f^{k})\sqrt{p}\log^{2k}{p}\right) \\
&= O(H^{\frac{k}{2}}(d^{2k}+2d^{k}d_g^{k}+2d^{k}d_f^{k})\sqrt{p}\log^{2k}{p}).
\end{align*}
This proves \eqref{Mkodd}.

Next, if $k$ is even, we use Lemma \ref{lem:SJJ} for $S(k/2,k/2)$ and Lemma \ref{lem:SJ12} for other terms. We obtain
\begin{align*}
M_k &= \frac{1}{2^k((\beta-\alpha)H)^{\frac{k}{2}}}\binom{k}{k/2}S(\frac{k}{2},\frac{k}{2})+O(2^\frac{k}{2}H^{\frac{k}{2}}(d^{2k}+2d^{k}d_g^{k}+2d^{k}d_f^{k})\sqrt{p}\log^{2k}{p}) \\
&= \frac{1}{2^k}\binom{k}{k/2}(k/2)!\abs{\CI}(1+O(k^2/H)) \\
&\qquad +O(2^{\frac{k}{2}}H^{\frac{k}{2}}(d^{2k}+2d^{k}d_g^{k}+2d^{k}d_f^{k})\sqrt{p}\log^{2k}{p}).
\end{align*}
This proves \eqref{Mkeven} and hence finished the proof of Theorem \ref{mainthm}.

\section{Proof of Corollary \ref{corGD}} \label{pfcorGD}

From Corollary \ref{maincor}, we obtain the limit
\begin{equation*}
\lim_{p\rightarrow\infty}\frac{2^{k/2}M_k}{\abs{\CI}}=\mu_k,
\end{equation*}
where
\begin{equation*}
\mu_k=\begin{cases}
1 \cdot 3 \cdot \ldots \cdot (k-1), & \text{if $k$ is even},\\
0, & \text{if $k$ is odd}.
\end{cases}
\end{equation*}
From the definition of $M_k$, this is
\begin{equation}\label{eqnmu}
\lim_{p\rightarrow\infty}\frac{1}{\abs{\CI}}\sum_{n\in\CI}(\sqrt{2}u_n)^k = \mu_k.
\end{equation}

Let $N_p(s)$ be the number of integers $n\in\CI$ such that $u_n\leq s$. Then $N_p(s)$ is a monotonic increasing step-function of $s$, with discontinuities at $s=s_1,s_2,\ldots,s_h$, say. Note that $N_p(s)=0$ if $s<-H$, and $N_p(s)=\abs{\CI}$ if $s\geq H$. Collecting together the values of $n\in\CI$ for which $u_n=s_i$ in \eqref{eqnmu}, we get (set $N_p(s_0)=0$ by convention)
\begin{equation*}
\lim_{p\rightarrow\infty}\frac{1}{\abs{\CI}}\sum_{i=1}^h(\sqrt{2}s_i)^k(N_p(s_i)-N_p(s_{i-1})) = \mu_k.
\end{equation*}
The left hand side of the above equation can be written as a Riemann-Stieltjes integral
\[\ds \text{LHS}=\int_{-\infty}^{\infty}(\sqrt{2}t)^k\,d\phi_p(t), \]
where
\[\phi_p(t)=\frac{1}{\abs{\CI}}N_p(s).\]
Set
\begin{equation*}
\phi(t)=\frac{1}{\sqrt{\pi}}\int_{-\infty}^{t}e^{-u^2}\,du,
\end{equation*}
then we have
\begin{align*}
\int_{-\infty}^{\infty}(\sqrt{2}t)^k\,d\phi(t) &=\frac{1}{\sqrt{\pi}}\int_{-\infty}^{\infty}(\sqrt{2}t)^ke^{-t^2}\,dt \\
&= \frac{1}{\sqrt{\pi}}2^{\frac{k}{2}}(1+(-1)^k)\Gamma((1+k)/2) \\
&= \mu_k.
\end{align*}
Thus
\begin{equation}\label{eqmom}
\lim_{p\rightarrow\infty}\int_{-\infty}^{\infty}(\sqrt{2}t)^k\,d\phi_p(t) = \int_{-\infty}^{\infty}(\sqrt{2}t)^k\,d\phi(t)
\end{equation}
for any $k$. Essentially by the uniqueness of the moment problem with bounded support in probability theory (see for example \cite{Fel68}), one can deduce from \eqref{eqmom} that
\[\ds \lim_{p\rightarrow\infty}\phi_p(t) = \phi(t).\]
This finishes the proof of Corollary \ref{corGD}.

\section{The case for $\chi$ trivial}

The case when $\chi$ is the trivial character is easy to settle. Our Theorem \ref{mainthm} do not make any assumptions on $g(x)$, hence it is easy to see that the theorem still hold when $\chi$ is the trivial character, if other conditions in the theorem is still assumed. Indeed, given the exponential sum
\begin{equation*}
S_n = \sum_{\substack{n<x(P_i)\leq n+H \\ y(P_i)\in\CJ}}\psi(f(P_i)),
\end{equation*}
we can form another hybrid sum
\begin{equation*}
S'_n = \sum_{\substack{n<x(P_i)\leq n+H \\ y(P_i)\in\CJ}}\chi(g(P_i))\psi(f(P_i)),
\end{equation*}
with $\chi$ being the quadratic character, and $g(x,y)=x^2$. Then $S_n$ and $S'_n$ have the same values unless there is a $P_i$ with zero $x$-coordinate, and the number of such $P_i$ is at most $\deg{(X)}$, which is much smaller than the error term in Theorem \ref{mainthm}. Now we apply the theorem to $S'_n$, and get the same conclusion for $S_n$. Since the corollaries to Theorem \ref{mainthm} does not make use of the properties of characters, they will continue to hold once Theorem \ref{mainthm} is correct. In particular, we have Gaussian distribution for the limiting distribution of the values of $S_n$ also when $\chi$ is trivial. This finishes the proof of Theorem \ref{thmtrivialchi}.

\section{The case for $\psi$ trivial}

The case when $\psi$ is the trivial character is a little bit more subtle, since Lemma \ref{lem:f} is not applicable. We return to the calculation of the moments $M_k$ in \eqref{eqnMk}, and investigate the sum $S(j_1,j_2)$ in \eqref{eqnSJ12}. If $\psi$ is trivial, then \eqref{calSJ12} becomes
\begin{equation*}
S(j_1,j_2) = \sum_{h_1=1}^H\cdots\sum_{h_{j_1+j_2}=1}^H\sum_{\substack{x\in\CI, y_1,\ldots,y_{j_1+j_2}\in\CJ \\ (x,y_1,\ldots,y_{j_1+j_2})\in X_{\CU_{\textbf{h}}}}}\chi(\tilde{g}(x,y_1,\ldots,y_{j_1+j_2})),
\end{equation*}
where
\begin{equation*}
\tilde{g}_{\textbf{h}}(x,y_1,\ldots,y_{j_1+j_2}) = \frac{g(x+h_1,y_1)\ldots g(x+h_{j_1},y_{j_1})}{g(x+h_{j_1+1},y_{j_1+1})\ldots g(x+h_{j_1+j_2},y_{j_1+j_2})}.
\end{equation*}
We recall that $y_i$ and $y_j$ stand for the same indeterminate if and only if $h_i=h_j$.

Let $a$ be the order of $\chi$. We can apply Lemma \ref{lem:perel} if $\tilde{g}$ is not a complete $a$-th power thanks to the assumptions in Theorem \ref{thmtrivialpsi} and Remark \ref{rmkperel}. From the assumption we made to $g$, and that $H$, $\deg{(g)}$ are small compared to $p$, we see that products and quotients of $g(x+h_i,y_i)$ with distinct $y_i$'s cannot be a complete $a$-th power (even when $g$ does not depend on $y$) in any irreducible component of the $x$-shifted curve $X_{\CU_{\textbf{h}}}$. Hence, if the $g(x+h_i,y_i)$'s stack together and become a complete $a$-th power, it must come from $a$ terms with the same $h_i$, or the same $h_i$ appears in both the numerator and denominator of $\tilde{g}$.

Suppose first that $j_1-j_2$ is not a multiple of $a$, then from the above discussion we see that $\tilde{g}$ can never be a complete $a$-th power. Hence we can use Lemma \ref{lem:perel} to obtain the estimate
\begin{equation}\label{eqnnotmult}
S(j_1,j_2) = O(H^{j_1+j_2}(d^{2j_1+2j_2}+2d^{j_1+j_2}d_g^{j_2}+2d^{j_1+j_2}d_f^{j_1+j_2})\sqrt{p}\log^{2j_1+2j_2}{p})
\end{equation}
for those terms.

If $j_1-j_2=ma$ for some integer $m$, then we may obtain an $a$-th power by having the same $h_i$ in the numerators and denominators, and group the remaining terms into $\abs{m}$ blocks, each block consists of $a$ terms with the same $h_i$. Note that for $a=2$ these two ways are the same since $\chi(g(x+h_i,y_i))$ agrees with its reciprocal. Fixing $j_1$ and $j_2$, it is easy to count the total number of such terms that make $\tilde{g}$ a complete $a$-th power. Letting $j=\min\{j_1,j_2\}$, this number is
\[\ds j!\frac{(\abs{m}a)!}{(a!)^{\abs{m}}\abs{m}!}H^{j+\abs{m}}(1+O(j^2/H)) \]
when $a>2$, and is
\[\ds \frac{(j_1+j_2)!}{2^{\frac{j_1+j_2}{2}}\left(\frac{j_1+j_2}{2}\right)!}H^{\frac{j_1+j_2}{2}}(1+O((j_1+j_2)^2/H)) \]
for $a=2$. Now each of the terms contribute at most
\[\abs{\CI}(\beta-\alpha)^{j+\abs{m}}+O(2^{j+\abs{m}}d^{2(j+\abs{m})}\sqrt{p}\log^{j+\abs{m}}{p}) \]
to the sum for $a>2$, and
\[\abs{\CI}(\beta-\alpha)^{\frac{j_1+j_2}{2}}+O(2^{\frac{j_1+j_2}{2}}d^{j_1+j_2}\sqrt{p}\log^{\frac{j_1+j_2}{2}}{p}) \]
to the sum for $a=2$. Hence, the sum of all terms that we cannot apply Lemma \ref{lem:perel} is
\[j!\frac{(\abs{m}a)!}{(a!)^{\abs{m}}\abs{m}!}H^{j+\abs{m}}\abs{\CI}(\beta-\alpha)^{j+\abs{m}}(1+O(j^2/H)) +O(2^{j+\abs{m}}d^{2(j+\abs{m})}\sqrt{p}\log^{j+\abs{m}}{p}) \]
for $a>2$, and is
\begin{multline*}
\frac{(j_1+j_2)!}{2^{\frac{j_1+j_2}{2}}\left(\frac{j_1+j_2}{2}\right)!}H^{\frac{j_1+j_2}{2}} \abs{\CI}(\beta-\alpha)^{\frac{j_1+j_2}{2}}(1+O((j_1+j_2)^2/H)) \\
+O(2^{\frac{j_1+j_2}{2}}d^{j_1+j_2}\sqrt{p}\log^{\frac{j_1+j_2}{2}}{p})
\end{multline*}
for $a=2$.

For the other terms, we use Lemma \ref{lem:perel}, and the contribution of these terms to the sum is
\[O(H^{2(j+\abs{m})}(d^{4(j+\abs{m})}-2d^{2(j+\abs{m})}d_{g}^{j+\abs{m}}+2d^{2(j+\abs{m})}d_{f}^{2(j+\abs{m})})\sqrt{p}(2\log{p}+1)^{2(j+\abs{m})}) \]
for $a>2$, and
\[O(H^{j_1+j_2}(d^{2(j_1+j_2)}-2d^{j_1+j_2}d_{g}^{\frac{j_1+j_2}{2}}+2d^{j_1+j_2}d_{f}^{j_1+j_2})\sqrt{p}(2\log{p}+1)^{j_1+j_2}) \]
for $a=2$.

Combining the above estimations, we finally get
\begin{multline}\label{eqnismult}
S(j_1,j_2) = j!\frac{(\abs{m}a)!}{(a!)^{\abs{m}}\abs{m}!}H^{j+\abs{m}}\abs{\CI}(\beta-\alpha)^{j+\abs{m}}(1+O(j^2/H)) +O(2^{j+\abs{m}}H^{2(j+\abs{m})} \\
\times(d^{4(j+\abs{m})}-2d^{2(j+\abs{m})}d_{g}^{j+\abs{m}}+2d^{2(j+\abs{m})}d_{f}^{2(j+\abs{m})})\sqrt{p}\log^{2(j+\abs{m})}{p}),
\end{multline}
for $j_1-j_2=ma$, $j=\min\{j_1,j_2\}$ and $a>2$, and
\begin{multline*}
S(j_1,j_2) = \frac{(j_1+j_2)!}{2^{\frac{j_1+j_2}{2}}\left(\frac{j_1+j_2}{2}\right)!}H^{\frac{j_1+j_2}{2}} \abs{\CI}(\beta-\alpha)^{\frac{j_1+j_2}{2}}(1+O((j_1+j_2)^2/H)) \\
+O(2^{\frac{j_1+j_2}{2}}H^{j_1+j_2}(d^{2(j_1+j_2)}-2d^{j_1+j_2}d_{g}^{\frac{j_1+j_2}{2}}+2d^{j_1+j_2}d_{f}^{j_1+j_2}) \\
\times\sqrt{p}(2\log{p}+1)^{j_1+j_2})
\end{multline*}
for $j_1=j_2=ma$, $a=2$. Note that $S(j_1,j_2)$ only depend on $j_1+j_2$ but not the particular $j_1,j_2$.

With the estimations for $S(j_1,j_2)$ in hand, we are ready to calculate the moments. From \eqref{eqnMk} we have
\begin{equation*}
M_k =\frac{1}{2^k((\beta-\alpha)H)^{\frac{k}{2}}}\sum^k_{j=0}\binom{k}{j}e^{(k-2j)i\theta}S(j,k-j).
\end{equation*}
There are $2$ cases according to the parity of $a$.

\subsection{The case when $a$ is even}

First suppose $a$ is even. Then if $k$ is odd, we have $j-(k-j)=2j-k$ is also odd, and thus it can never be a multiple of $a$. Every term in the above sum can then be estimated using \eqref{eqnnotmult}. We have
\begin{align*}
M_k &= O\left(\frac{1}{2^k((\beta-\alpha)H)^{\frac{k}{2}}}\sum^k_{j=0}\binom{k}{j}H^{k}(d^{2k}+2d^{k}d_g^{k-j}+2d^{k}d_f^{k})\sqrt{p}\log^{2k}{p}\right) \\
&= O(H^{\frac{k}{2}}(d^{2k}+2d^{k}d_g^{k}+2d^{k}d_f^{k})\sqrt{p}\log^{2k}{p}).
\end{align*}
This proves \eqref{2Mkodd} and \eqref{aMkodd} in Theorem \ref{thmtrivialpsi}.

If $k$ is even, then the case for $a=2$ is different from the others. If $a=2$, we have
\begin{align*}
M_k &= \frac{1}{2^k((\beta-\alpha)H)^{\frac{k}{2}}}\sum^k_{j=0}\binom{k}{j}e^{(k-2j)i\theta}S(j,k-j) \\
&= \frac{1}{2^k((\beta-\alpha)H)^{\frac{k}{2}}}\sum^k_{j=0}\binom{k}{j}e^{(k-2j)i\theta}\frac{k!}{2^{\frac{k}{2}}\left(\frac{k}{2}\right)!}H^{\frac{k}{2}} \abs{\CI}(\beta-\alpha)^{\frac{k}{2}}(1+O(k^2/H)) \\
&\qquad +O(2^{\frac{k}{2}}H^{k}(d^{2k}-2d^{k}d_{g}^{\frac{k}{2}}+2d^{k}d_{f}^{k})\sqrt{p}(2\log{p}+1)^{k}) \\
&= \frac{1}{2^k}\frac{k!}{2^{\frac{k}{2}}\left(\frac{k}{2}\right)!}\abs{\CI}\left(\sum^k_{j=0}\binom{k}{j}e^{(k-2j)i\theta}\right)(1+O(k^2/H)) \\
&\qquad +O(2^{\frac{k}{2}}H^{k}(d^{2k}-2d^{k}d_{g}^{\frac{k}{2}}+2d^{k}d_{f}^{k})\sqrt{p}(2\log{p}+1)^{k})
\end{align*}

In general we are unable to handle the sum $\sum^k_{j=0}\binom{k}{j}e^{(k-2j)i\theta}$, and we do not have Gaussian distribution for general $\theta$. See Remark \ref{rmkreal} for details.

For $\theta=0$, the above calculation of $M_k$ becomes
\begin{equation*}
M_k = \frac{k!}{2^{\frac{k}{2}}\left(\frac{k}{2}\right)!}\abs{\CI}(1+O(k^2/H)) +O(2^{\frac{k}{2}}H^{k}(d^{2k}-2d^{k}d_{g}^{\frac{k}{2}}+2d^{k}d_{f}^{k})\sqrt{p}(2\log{p}+1)^{k}).
\end{equation*}
This is \eqref{2Mkeven} in Theorem \ref{thmtrivialpsi}.

If $a>2$ (still even), then it is easy to see that $2j-k$ is a multiple of $a$ if and only if $j=k/2+m(a/2)$ for $m=-[k/a], -[k/a]+1, \ldots, [k/a]$, where $[x]$ denotes the greatest integer function. Note that for those $j$, we have $2j-k=ma$. By estimating the terms corresponding to above $j$ using \eqref{eqnismult}, and all other terms using \eqref{eqnnotmult}, we have
\begin{align*}
M_k &= \frac{1}{2^k((\beta-\alpha)H)^{\frac{k}{2}}} \sum_{m=-[\frac{k}{a}]}^{[\frac{k}{a}]}\frac{k!}{(\frac{k}{2}+\frac{a\abs{m}}{2})!} \frac{(\abs{m}a)!}{(a!)^{\abs{m}}}\abs{\CI}(H(\beta-\alpha))^{(\frac{k}{2}-a\frac{\abs{m}}{2})+\abs{m}} \\
&\qquad\qquad \times(1+O(k^2/H)) +O(H^{\frac{3k}{2}}(d^{4k}-2d^{2k}d_{g}^{k}+2d^{2k}d_{f}^{2k})\sqrt{p}\log^{2k}{p}) \\
&= \frac{1}{2^k}\frac{k!}{(k/2)!}\abs{\CI} (1+O(k^{\frac{a}{2}+2}/H))+O(H^{\frac{3k}{2}}(d^{4k}-2d^{2k}d_{g}^{k}+2d^{2k}d_{f}^{2k})\sqrt{p}\log^{2k}{p}).
\end{align*}
This is \eqref{aMkeven}.

\subsection{The case when $a$ is odd}

Let $a$ be an odd integer, $a>1$. Again there are two cases according to the parity of $k$. First assume $k$ is even, then $2j-k$ is even, and since $a$ is odd, $2j-k$ will be a multiple of $a$ if and only if it is a multiple of $2a$. Therefore, the result here is the same as the case where the order of $\chi$ is $2a$. We have
\begin{multline*}
M_k = \frac{1}{2^k}\frac{k!}{(k/2)!}\abs{\CI} (1+O(k^{\frac{a}{2}+2}/H)) \\
+O(H^{\frac{3k}{2}}(d^{4k}-2d^{2k}d_{g}^{k}+2d^{2k}d_{f}^{2k})\sqrt{p}\log^{2k}{p}).
\end{multline*}
This gives \eqref{oMkeven}.

Now if $k$ is odd, then if $2j-k=ma$, $m$ must be odd. Similar to the calculation in the case $a$ even and $a>2$, we see that the main terms of $M_k$ correspond to $j=\frac{k}{2}+\frac{a}{2}$ and $j=\frac{k}{2}-\frac{a}{2}$. We have
\begin{align*}
M_k &= \frac{1}{2^k((\beta-\alpha)H)^{\frac{k}{2}}}\binom{k}{\frac{k}{2}-\frac{a}{2}}(e^{ia\theta}+e^{-ia\theta}) \left(\frac{k}{2}-\frac{a}{2}\right)!H^{\frac{k}{2}-\frac{a}{2}+1}\abs{\CI}(\beta-\alpha)^{\frac{k}{2}-\frac{a}{2}+1} \\
&\qquad \times (1+O(k^{a+2}/H))+O((H^{\frac{3k}{2}}(d^{4k}-2d^{2k}d_{g}^{k}+2d^{2k}d_{f}^{2k})\sqrt{p}\log^{2k}{p}) \\
&= \frac{1}{2^k((\beta-\alpha)H)^{\frac{a}{2}-1}}\frac{k!}{\left(\frac{k}{2}+\frac{a}{2}\right)!}(2\cos{a\theta})\abs{\CI}(1+O(k^{a+2}/H)) \\
&\qquad +O((H^{\frac{3k}{2}}(d^{4k}-2d^{2k}d_{g}^{k}+2d^{2k}d_{f}^{2k})\sqrt{p}\log^{2k}{p}).
\end{align*}
This proves \eqref{oMkodd} and finishes the proof of Theorem \ref{thmtrivialpsi}.

%\bibliography{hyexp10}
%\bibliographystyle{amsplain}

\end{document}